\documentclass[12pt]{amsart}

\usepackage{amssymb,amsmath,amscd,amsthm,enumerate,verbatim}
\usepackage[utf8]{inputenc}
\usepackage{mathrsfs}
\usepackage{bbm}
\usepackage[left=1.3in,right=1.3in,top=1in,bottom=1in]{geometry}
\usepackage{bm}

\date{}

\newcommand{\bbC}{{\mathbbm C}}
\newcommand{\bbL}{{\mathbbm{L}}}
\newcommand{\bbN}{{\mathbbm N}}
\newcommand{\bbR}{{\mathbbm R}}
\newcommand{\bbT}{{\mathbbm T}}
\newcommand{\bbZ}{{\mathbbm Z}}

\newcommand{\scrF}{{\mathscr{F}}}

\newcommand{\scrH}{{\mathscr{H}}}

\newcommand{\boldtheta}{{\bm{\theta}}}

\newtheorem{theorem}{Theorem}[section]
\newtheorem{lemma}[theorem]{Lemma}
\newtheorem{prop}[theorem]{Proposition}
\newtheorem{coro}[theorem]{Corollary}

\theoremstyle{definition}

\theoremstyle{definition}

\theoremstyle{definition}

\theoremstyle{definition}

\theoremstyle{definition}

\allowdisplaybreaks

\numberwithin{equation}{section}

\newcommand{\set}[1]{\left\{#1\right\}}

\begin{document}

\title[Ballistic Transport for Periodic Operators]{Ballistic Transport for  Periodic \\ Jacobi Operators on $\bbZ^d$}

\author[J.\ Fillman]{Jake Fillman}

\thanks{J.\ F.\ was supported in part by Simons Foundation Collaboration Grant \#711633. }

\maketitle

\begin{abstract}
In this expository work, we collect some background results and give a short proof of the following theorem: periodic Jacobi matrices on $\bbZ^d$ exhibit strong ballistic motion.
\end{abstract}


\section{Introduction}

This expository note is concerned with the properties of Jacobi operators on $\bbZ^d$. Concretely, we fix the \emph{dimension} $d \in \bbN$ and consider linear operators $J=J_{a,b}:\ell^2(\bbZ^d) \to\ell^2(\bbZ^d)$ given by
\[[Ju]_{\bm{x}} = \sum_{\substack{\bm{y} \in \bbZ^d \\ \|\bm{x}-\bm{y}\|_1=1}} a_{\bm{x},\bm{y}}u_{\bm{y}} + b_{\bm{x}} u_{\bm{x}},
\quad u \in \ell^2(\bbZ^d), \ \bm{x} \in \bbZ^d,\]
where $a$ and $b$ are bounded, $b$ is real-valued, and $a_{\bm{y},\bm{x}} = a_{\bm{x},\bm{y}}^*\neq 0$ for all $\bm{x}$ and $\bm{y}$ in $\bbZ^d$ for which $\|\bm{x}-\bm{y}\|_1=1$.
For convenience, we will write $\bm{x} \sim \bm{y}$ for $\bm{x}, \bm{y} \in \bbZ^d$ to mean $\|\bm{x}-\bm{y}\|_1=1$.

In the case $d = 1$, we simply write $a_x := a_{x,x+1}$ and one obtains the familiar \emph{Jacobi matrix} on $\ell^2(\bbZ)$:
\[J = \begin{bmatrix}
\ddots & \ddots & \ddots \\
& a_0^* & b_1 & a_1 \\
&& a_1^* & b_2 & a_2 \\
&&&\ddots & \ddots & \ddots
\end{bmatrix}.\]
Jacobi matrices have inspired intense study over the years owing to their close connections with orthogonal polynomials, integrable systems, and mathematical physics; see, e.g., \cite{Dam2017ETDS, MarxJito2017ETDS, Simon2011:Szego, Teschl2000:Jacobi} and references therein.

We will be interested in the case in which $J$ is \emph{periodic}, i.e., there exists a full-rank subgroup $\bbL \subseteq \bbZ^d$ such that  $U^{\bm\ell}JU^{-\bm\ell}=J$ for all $\bm\ell \in \bbL$, where $U^{\bm\ell}$ denotes the shift $\delta_{\bm x} \mapsto \delta_{\bm x + \bm \ell}$. 
Equivalently,
\begin{equation} \label{eq:perVdef} 
a_{\bm{x}+ \bm{\ell}, \bm{y}+\bm{\ell}} = a_{\bm{x}, \bm{y}} \text{ and } b_{\bm{x}+ \bm{\ell}} = b_{\bm{x}} \text{ for all }\bm x \sim \bm{y} \in \bbZ^d \text{ and all }\bm{\ell} \in \bbL.\end{equation}
Of course, if $\{\bm{\ell}_1,\ldots,\bm{\ell}_d\}$ are linearly independent elements of $\bbZ^d$ generating a lattice $\bbL$ for which \eqref{eq:perVdef} holds, then it is a straightforward calculation to show that \eqref{eq:perVdef}  also holds for the lattice $\bbL' = r \bbZ^d = \{r \bm{n} : \bm n \in \bbZ^d\}$, where $r = |\det(\bm{\ell}_1 | \cdots | \bm{\ell}_d)|$ is the volume of $\bbR^d/\bbL$; consequently, no real generality is lost in considering lattices generated by multiples of the standard basis vectors, so we shall consider precisely this scenario in the present note. To that end, given $\bm{q} \in \bbN^d$, say $J$ is $\bm{q}$-periodic if $U^{\bm{mq}}JU^{-\bm{mq}} = J$ for all $\bm{m} \in \bbZ^d$, where 
\[ \bm{mq} = (m_1q_1, m_2q_2,\ldots,m_dq_d). \]
Equivalently, $J$ is $\bm{q}$ periodic if \eqref{eq:perVdef}  holds for the lattice $\bbL =  \bigoplus_{j=1}^d q_j \bbZ$.

The goal of the note is to discuss quantum dynamics associated with such periodic operators. In particular, we focus on the growth of the position observables. For a linear operator $O$, we denote by $O(t) = e^{itJ}O e^{-itJ}$ the corresponding time evolution with respect to $J$. For $1 \le j \le d$, the $j$th position operator  $X_j$ is given by $X_j\delta_{\bm{x}} = x_j\delta_{\bm{x}}$, where
\[D(X_j) = \set{\psi \in \ell^2(\bbZ^d) : X_j \psi \in \ell^2(\bbZ^d)}.\] The \emph{vector} position operator $\bm{X} :D(\bm{X}) = \bigcap_{j=1}^d D(X_j) \to \ell^2(\bbZ^d) \otimes \bbC^d$ given by
\[\bm{X}\psi = (X_1\psi,\ldots,X_d\psi).\]

The primary phenomenon that we will discuss is that of \emph{ballistic motion}, i.e., linear growth of the position observable(s). This was established for continuum Schr\"odinger operators by Asch--Knauf \cite{AscKna1998Nonlin} and was later extended to Jacobi matrices in $d=1$ by Damanik--Lukic--Yessen \cite{DamLukYes2015CMP}. The result we want to discuss is the generalization of \cite{DamLukYes2015CMP} to the case of general $d \geq 1$.

\begin{theorem} \label{t:ball}
If $J$ is periodic, then it exhibits ballistic motion in the following sense. There are bounded, self-adjoint operators $Q_k$, $1\le k \le d$, such that
\[\lim_{t \to \infty} \frac{X_k(t)}{t}  = Q_k\]
in the strong sense, and $\ker(Q_k) = \{0\}$. In particular,
\[\bm{Q} := \lim_{t\to\infty} \frac{\bm{X}(t)}{t}\]
in the strong sense and $\ker(\bm{Q}) = \{0\}$.
\end{theorem}

Naturally, since Asch--Knauf already worked in higher dimension, the result is not surprising and could indeed be considered a folklore result, since it is simply a convex combination of \cite{AscKna1998Nonlin} and \cite{DamLukYes2015CMP}. Nevertheless, we felt it would be worthwhile to have an essentially self-contained exposition of the proof somewhere in the literature.

In recent years, there has also been substantial interest in studying the phenomenon of ballistic motion in specific aperiodic models. For instance, this has been established for limit-periodic and quasi-periodic models \cite{Fillman2017CMP, GeKach2020Preprint, Kachkovskiy2016CMP, Kachkovskiy2020Preprint, KarpLeeShte2017CMP, ZhangZhao2017CMP}.

In Section~\ref{sec:floq} we discuss a direct integral decomposition of $J$, and then explain in Section~\ref{sec:ball} how to use this to prove Theorem~\ref{t:ball}. Since the paper is expository in nature, we aim to supply proofs so that the article is self-contained, modulo background facts from functional analysis and analytic perturbation theory.

\subsection*{Acknowledgements}

J.F. is grateful to Ilya Kachkovskiy and Milivoje Lukic for helpful conversations. Work supported in part by Simons Foundation Collaboration Grant \#711663.

\section{Decomposition of $J$} \label{sec:floq}

We first explain how to decompose $J$ as a direct integral of operators on the fundamental domain with suitable self-adjoint boundary conditions.
The reader is referred to \cite{ReedSimon4} for additional background about direct integrals.
 More precisely, let $\bbT^d = \bbR^d/\bbZ^d$, put 
\[\Gamma = \bbZ^d \cap \prod_{j=1}^d [0,q_j) = \{0,1,\ldots,q_1-1\} \times \cdots \times \{0,1,\ldots,q_d-1\},\]
and consider $\scrH(\boldtheta) = \scrH(\boldtheta,\bm q) \subset \ell^\infty(\bbZ^d)$ comprising all those $\psi:\bbZ^d \to \bbC$ such that 
\begin{equation}
\psi_{\bm{x}+\bm{nq}} = e^{2\pi i \langle \boldtheta,\bm{n} \rangle}\psi_{\bm{x}}.
\end{equation}
With the inner product
\[\langle \psi,\varphi\rangle_{\scrH(\boldtheta)} = \sum_{\bm{x} \in \Gamma} \overline{\psi_{\bm{x}}} \varphi_{\bm{x}},\]
$\scrH(\boldtheta)$ becomes a Hilbert space of dimension $\bar q = \#\Gamma = \prod_{j=1}^d q_j$.
We will use $d\boldtheta$ to denote the Lebesgue measure on $\bbT^d$.

\begin{lemma}
If $J$ is $\bm{q}$-periodic, then $J$ maps $\scrH(\boldtheta,\bm{q})$ into itself for all $\boldtheta$.
\end{lemma}

\begin{proof}
This is a short calculation. If $\psi \in \scrH(\boldtheta)$, then
\begin{align*}
[J\psi]_{\bm{x}+\bm{nq}}
& =  \sum_{\bm{x} \sim\bm{y}} a_{\bm{x}+\bm{nq},\bm{y}+\bm{nq}} \psi_{\bm{y}+\bm{nq}} + b_{\bm{x}+\bm{nq}} \psi_{\bm{x}+\bm{nq}} \\
& =  e^{2\pi i \langle \boldtheta,\bm{n} \rangle}\sum_{\bm{x} \sim\bm{y}} a_{\bm{x},\bm{y}} \psi_{\bm{y}} + e^{2\pi i \langle \boldtheta,\bm{n} \rangle} b_{\bm{x}} \psi_{\bm{x}} \\
& =e^{2\pi i \langle \boldtheta,\bm{n} \rangle}[J\psi]_{\bm{x}},
\end{align*}
whence $J\psi \in \scrH(\boldtheta)$.
\end{proof}

In view of the lemma, we may define $J(\boldtheta) = J|_{\scrH(\boldtheta)}$ for each $\boldtheta \in \bbT^d$. 
Writing $\bbC^\Gamma$ for the space of functions $\Gamma \to \bbC$, one can view $\scrH(\boldtheta) \cong \bbC^\Gamma$ via the identification
\begin{equation} \label{eq:scrhthetatoCGamma} \bbC^\Gamma \ni \delta_{\bm{x}} \mapsto \sum_{\bm{n} \in \bbZ^d} e^{2\pi i \langle \boldtheta,\bm{n} \rangle} \delta_{\bm{x}+\bm{nq}} \in \scrH(\boldtheta),\end{equation}
so we also freely consider $J(\boldtheta)$ as a linear operator on $\bbC^\Gamma$. 

To describe the decomposition of $J$, define 
\begin{align*} \scrH_1 & = \int_{\bbT^d}^\oplus \scrH(\boldtheta) \, d\boldtheta,\end{align*}
which consists of measurable functions $f$ mapping $\bbT^d$ into $\bigcup_{\boldtheta \in \bbT^d} \scrH(\boldtheta)$ such that $f(\boldtheta) \in \scrH(\boldtheta)$ for all $\boldtheta$ and
\[\|f\|_{\scrH_1}^2 := \int_{\bbT^d} \|f(\boldtheta) \|^2_{\scrH(\boldtheta)} \, d\boldtheta < \infty.\] 
Equipped with the inner product
\[\langle f, g \rangle_{\scrH_1} = \int_{\bbT^d} \langle f(\boldtheta), g(\boldtheta) \rangle_{\scrH(\boldtheta)} \, d\boldtheta,\]
$\scrH_1$ is a Hilbert space; see \cite{ReedSimon4} for details.
Write $f(\boldtheta,\bm{x})$ for the $\bm{x}$th coordinate of $f(\boldtheta)$. Identifying the fibers of $\scrH_1$ with $\bbC^\Gamma$ as in \eqref{eq:scrhthetatoCGamma}, we can also view $\scrH_1$ simply as the collection of square-integrable maps $\bbT^d \to \bbC^\Gamma$, which we shall do freely when it is convenient to do so.

For $\psi \in \ell^1(\bbZ^d)$, define
\[[\scrF\psi](\boldtheta,\bm{x}) = \sum_{\bm{m} \in \bbZ^d} \psi_{\bm{x}+\bm{mq}} e^{2 \pi i\langle \boldtheta, \bm{m} \rangle}. \]
\begin{lemma}
For every $\psi \in \ell^1(\bbZ^d)$, $\scrF\psi \in \scrH_1$, $\|\scrF \psi\|_{\scrH_1} = \|\psi\|_{\ell^2(\bbZ^d)}$, and the image of $\ell^1(\bbZ^d)$ is dense in $\scrH_1$. In particular, $\scrF$ extends to a unitary operator $\scrF :\ell^2(\bbZ^d) \to \scrH_1$.
\end{lemma}

\begin{proof}
For $\bm{x} \in \Gamma$ and $\bm{n} \in \bbZ^d$, denote $\scrF \delta_{\bm{x}+\bm{nq}} = \varphi_{\bm{x},\bm{n}}$ and note that
\begin{equation} \label{eq:varphinx}\varphi_{\bm{x},\bm{n}}(\boldtheta) = \sum_{\bm{m} \in \bbZ^d} e^{2 \pi  i \langle \boldtheta, \bm{n}-\bm{m} \rangle} \delta_{\bm{x+\bm{mq}}}.\end{equation}
Since $\set{e^{2\pi i \langle \cdot, \bm{n} \rangle} : \bm{n} \in \bbZ^d}$ is an orthonormal basis of $L^2(\bbT^d)$, one can check that $\set{\varphi_{\bm{x},\bm{n}}: \bm{x} \in \Gamma, \, \bm{n} \in \bbZ^d}$ is  an orthonormal basis of $\scrH_1$, so the lemma follows immediately.
\end{proof}

The unitary operator $\scrF$ ``diagonalizes'' $J$ in the sense that it transforms $J$ to a (matrix) multiplication operator given by pointwise multiplication by $J(\boldtheta)$ on $\scrH_1$. Concretely, define a linear operator $\widehat{J} : \scrH_1 \to \scrH_1$ by
\begin{equation} \label{eq:Jhatdirectint}
 [\widehat{J}g](\boldtheta) = J(\boldtheta) g(\boldtheta). \end{equation}
It is convenient to use the direct integral notation for operators enjoying a decomposition as in \eqref{eq:Jhatdirectint}; for instance, we will write
\[\widehat{J} = \int_{\bbT^d}^\oplus J(\boldtheta)\, d\boldtheta. \]

\begin{theorem} \label{t:floquet1}
$\widehat{J} = \scrF J \scrF^*$.
\end{theorem}

\begin{proof}
This follows from a direct calculation.
Recall $\varphi_{\bm{x},\bm{n}} = \scrF \delta_{\bm{x}+\bm{nq}}$ for $\bm{x} \in \Gamma$ and $\bm{n} \in \bbZ^d$. Since $\varphi_{\bm{x},\bm{n}}(\boldtheta) \in \scrH(\boldtheta)$ for each $\boldtheta \in \bbT^d$, \eqref{eq:varphinx} yields
\begin{align*} J(\boldtheta) \varphi_{\bm{x},\bm{n}}(\boldtheta) & = J \varphi_{\bm{x},\bm{n}}(\boldtheta) \\
& = \sum_{\bm{m} \in \bbZ^d} e^{2 \pi  i \langle \boldtheta, \bm{n}-\bm{m} \rangle}J \delta_{\bm{x+\bm{mq}}} \\
& = \sum_{\bm{m} \in \bbZ^d} e^{2 \pi  i \langle \boldtheta, \bm{n}-\bm{m} \rangle} \left( \sum_{\bm{z} \sim \bm{x}+\bm{mq}} a_{\bm{\bm{x}+\bm{mq}}, \bm{z}} \delta_{\bm{z}} + b_{\bm{x+\bm{mq}}}\delta_{\bm{x+\bm{mq}}} \right) \\
& = \sum_{\bm{m} \in \bbZ^d} e^{2 \pi  i \langle \boldtheta, \bm{n}-\bm{m} \rangle} \left( \sum_{\bm{y} \sim \bm{x}} a_{\bm{x}, \bm{y}} \delta_{\bm{y+\bm{mq}}} + b_{\bm{x}}\delta_{\bm{x+\bm{mq}}} \right) \\
& = \sum_{\bm{y} \sim \bm{x}} a_{\bm{x}, \bm{y}} \varphi_{\bm{y,\bm{n}}}(\boldtheta) + b_{\bm{x}}\varphi_{\bm{x}, \bm{n}} (\boldtheta).
\end{align*}
On the other hand, a direct calculation from the definitions yields
\begin{align*} [\scrF J \scrF^* \varphi_{\bm{x},\bm{n}}](\boldtheta) & = [\scrF J \delta_{\bm{x}+\bm{nq}}](\boldtheta) \\
& =  \left[\scrF \left(\sum_{\bm{y} \sim \bm{x}} a_{\bm x, \bm y}\delta_{\bm{y}+\bm{nq}} +  b_{\bm x}\delta_{\bm{x}+\bm{nq}}\right)\right](\boldtheta) \\
& =  \sum_{\bm{y} \sim \bm{x}} a_{\bm x, \bm y}\varphi_{\bm{y},\bm{n}}(\boldtheta) +  b_{\bm x}\varphi_{\bm{x},\bm{n}}(\boldtheta).
\end{align*}
Thus $\widehat{J}\varphi_{\bm{x},\bm{n}} = \scrF J \scrF^*\varphi_{\bm{x},\bm{n}}$ for all $\bm{x}$ and $\bm{n}$. Since $\{\varphi_{\bm{x},\bm{n}}: \bm{x} \in \Gamma, \, \bm{n} \in \bbZ^d\}$ is a basis of $\scrH_1$, the theorem is proved.
\end{proof}

\begin{coro}
Let $\bar{q} = \prod_{j=1}^d q_j$,  let $E_1(\boldtheta) \leq \cdots \leq E_{\bar q}(\boldtheta)$ denote the eigenvalues of $J(\boldtheta)$, and define
\[I_k = \{E_k(\boldtheta) : \boldtheta \in \bbT^d\}.\]
Then
\[\sigma(H)= \bigcup_{k=1}^{\bar{q}} I_k = \bigcup_{\boldtheta \in \bbT^d} \sigma(J(\boldtheta)).\]
\end{coro}

\begin{proof}
Since $E_k(\boldtheta)$ depends continuously on $\boldtheta$ for every $1 \le k \le \bar{q}$, this is an immediate consequence of Theorem~\ref{t:floquet1}.
\end{proof}

For later use, we note the following:

\begin{lemma} \label{lem:eigenvalues}
For a.e.\ $\boldtheta \in \bbT^d$ and each $1\le k \le d$, $\frac{\partial E_j}{\partial \theta_k}(\boldtheta)$ exists and is nonzero.
\end{lemma}

\begin{proof}
This follows from analytic eigenvalue perturbation theory \cite{Kato1980:PertTh}.
 We describe the broad strokes and leave the details to the reader.

Let $1 \le k \le d$, choose and fix $\theta_j \in \bbT$ for $j \neq  k$. For $s \in \bbT$, define $\boldtheta^{(k)}(s) \in \bbT^d$ by $\theta^{(k)}_j(s) = \theta_j$ for $j \neq k$ and $\theta_k^{(k)}(s)=s$. Extending into the complex plane, we have that \[A(s):=J(\boldtheta^{(k)}(s)), \quad s \in\bbC\] 
is an analytic family of matrices. As such, one can choose branches\footnote{Note that we use the different notation $\lambda$ to emphasize that in general, the analytic enumeration $\lambda_j(\boldtheta)$ need not necessarily coincide with the ordered enumeration $E_j(\boldtheta)$.} of the eigenvalues $\lambda_1(s),\ldots,\lambda_t(s)$ and the associated eigenprojections which are real-analytic functions of $s$. In fact, these will be holomorphic functions of $s \in \bbC$ away from a discrete set of points. We note that the multiplicity  $m_j$ of $\lambda_j$ is constant, again away from a discrete set. We refer the reader to Kato \cite{Kato1980:PertTh} for details about analytic perturbation theory. See especially \cite[Theorem~II.6.1]{Kato1980:PertTh}; see also \cite[Theorem~1.4.1 and Corollary~1.4.5]{Simon2015:CCA4}.

Since $\lambda_r$ is non-constant for each $1 \le r \le t$, it follows that $\partial E_j/\partial \theta_k$ exists and is nonzero for each $1 \le j \le \bar{q}$ and a.e.\ $\boldtheta \in \bbT^d$.

\end{proof}




\section{Ballistic Motion} \label{sec:ball}

Let us make a few observations. Formally, for each $1 \le k \le d$,
\[\frac{d}{dt} X_k(t) = iJe^{itJ}X_ke^{-itJ} -ie^{itJ}X_kJ e^{-itJ} = P_k(t),\]
where $P_k = i[J,X_k] = i(JX_k-X_kJ)$. Thus, 
\begin{equation} \label{eq:xkavinteg} \frac{X_k(t)}{t} = \frac{X_k}{t} + \frac{1}{t} \int_0^s P_k(s) \, ds,\end{equation}
so one wants to understand the time averages of $P_k$. Since the integral equation \eqref{eq:xkavinteg}  involves the unbounded operator $X_k$, some care is needed, so let us make this precise. For each $N \in \bbN$, let $X_{k,N}$ denote the bounded operator given by 
\[X_{k,N} \delta_{\bm{x}} = \begin{cases} x_k \delta_{\bm{x}} & |x_k| \leq N \\ N\delta_{\bm{x}} & \text{otherwise},\end{cases}\] 
and define $P_{k,N} = i[J,X_{k,N}]$.

\begin{prop}\label{prop:xkintegral}
Let $J$ be a bounded Jacobi matrix on $\ell^2(\bbZ^d)$. For all $k$ and $t$, $D(X_k(t)) = D(X_k)$ and one has
\begin{equation} \label{eq:xkintegraleq} X_k(t) \psi = X_k \psi + \int_0^t P_k(s) \psi \, ds.\end{equation}
\end{prop}

The following representation of $P_k$ in the standard basis will be helpful.

\begin{prop} \label{prop:Pkcalc}
For $\bm{x} \in \bbZ^d$,
\[P_k \delta_{\bm{x}} =i a_{\bm{x}, \bm{x}-\bm{e}_k} \delta_{\bm{x}-\bm{e}_k} - i a_{\bm{x},\bm{x}+\bm{e}_k} \delta_{\bm{x}+\bm{e}_k}.\]
In particular, $\|P_k\| \leq 2\|a\|_\infty$.
\end{prop}

\begin{proof}
For $\bm{x} \in \bbZ^d$, observe that
\begin{align*}
J\delta_{\bm{x}} =  \sum_{j=1}^d \left( a_{\bm{x},\bm{x}+\bm{e}_j}\delta_{\bm{x}+\bm{e}_j} + a_{\bm{x},\bm{x}-\bm{e}_j}\delta_{\bm{x}-\bm{e}_j} \right) + b_{\bm{x}} \delta_{\bm{x}},\end{align*}
so
\[X_k J\delta_{\bm{x}} =  \sum_{j=1}^d \left((x_k+\delta_{j,k}) a_{\bm{x}, \bm{x}+\bm{e}_j}\delta_{\bm{x}+\bm{e}_j} + (x_k-\delta_{j,k} )a_{\bm{x}, \bm{x}-\bm{e}_j}\delta_{\bm{x}-\bm{e}_j} \right) + x_k b_{\bm{x}} \delta_{\bm{x}}.\]
Subtracting this from $JX_k \delta_{\bm{x}} = x_k J\delta_{\bm{x}}$, we get
\[P_k \delta_{\bm{x}} = i a_{\bm{x}, \bm{x}-\bm{e}_k} \delta_{\bm{x}-\bm{e}_k} - i a_{\bm{x},\bm{x}+\bm{e}_k} \delta_{\bm{x}+\bm{e}_k},\]
as desired.
\end{proof}

\begin{proof}[Proof of Proposition~\ref{prop:xkintegral}]
This follows from the same argument as \cite[Theorem~2.1]{DamLukYes2015CMP}. 
Since $X_{k,N}$ is bounded, a direct calculation shows that \eqref{eq:xkintegraleq} holds with $X_k$ replaced by $X_{k,N}$, that is,
\begin{equation} \label{eq:xkintegraleqcutoff}   X_{k,N}(t) \psi = X_{k,N} \psi + \int_0^t P_{k,N}(s) \psi \, ds.\end{equation}
Since $X_{k,N} \to X_k$ and $P_{k,N} \to P_k$ strongly, \eqref{eq:xkintegraleq} follows immediately from \eqref{eq:xkintegraleqcutoff}, the uniform bound \[\|P_k\| \leq 2 \|a\|_\infty,\]
and dominated convergence.
\end{proof}

From the representation of $P_k$ in Proposition~\ref{prop:Pkcalc}, we can see that $P_k$ is $\bm{q}$-periodic whenever $J$ is $\bm{q}$-periodic, so it too defines operators $P_k(\boldtheta) := P_k|_{\scrH(\boldtheta)}$ for each $\boldtheta\in \bbT^d$.

\begin{theorem}
We have $\scrF P_k \scrF^* = \widehat{P}_k$, where $P_k = \int_{\bbT^d}^\oplus P_k(\boldtheta) \, d\boldtheta$, that is,
\[[\widehat{P}_k g] (\boldtheta) = P_k(\boldtheta) g(\boldtheta).\]
\end{theorem}

\begin{proof}
This is essentially the same calculation as in the proof of Theorem~\ref{t:floquet1}.
\end{proof}

We now have all the necessary pieces in place in order to prove the main result.

\begin{proof}[Proof of Theorem~\ref{t:ball}] For $\psi \in \bbC^\Gamma$, $\bm{x} \in \Gamma$, and $\boldtheta \in \bbT^d$, define \[\bm{x}/\bm{q} =(x_1/q_1,\ldots,x_d/q_d)\] and define the multiplication operator $M:\scrH_1 \to \scrH_1$ by $[Mg](\boldtheta) = M(\boldtheta)g(\boldtheta)$, where
\[[M(\boldtheta)\psi]_{\bm x} = e^{2\pi i \langle \boldtheta, \bm{x}/ \bm{q}\rangle} \psi_{\bm{x}}\]
Write $\widetilde{J}(\boldtheta) = M(\boldtheta)^{-1}J(\boldtheta) M(\boldtheta)$ and likewise for $\widetilde P_k(\boldtheta)$. A direct calculation shows
\[\widetilde P_k(\boldtheta)
=  \frac{q_k}{2\pi} \frac{\partial}{\partial \theta_k} \widetilde J(\boldtheta),\]
and thus, denoting the projection onto the eigenspace of $E_j(\boldtheta)$ by $\Pi_j(\boldtheta)$, we have
\[\Pi_j(\boldtheta)P_k(\boldtheta)\Pi_j(\boldtheta) = \frac{q_k}{2\pi}\frac{\partial E_j}{\partial \theta_k}(\boldtheta) \Pi_j(\boldtheta)\]
for a.e.\ $\boldtheta$ (by Lemma~\ref{lem:eigenvalues}).

Thus, we have 
\begin{align*}
\frac{X_k(t)}{t} & = \frac{X_k}{t} + \frac{1}{t} \int_0^t P_k(s) \, ds \\
& = \frac{X_k}{t} + \frac{1}{t} \int_0^t \int^\oplus_{\bbT^d} e^{isJ(\boldtheta)} P_k(\boldtheta) e^{-isJ(\boldtheta)} \, d\boldtheta \, ds \\
 &\to \frac{q_k}{2\pi} \int^\oplus_{\bbT^d}  \sum_{j=1}^{\bar q} \frac{\partial E_j}{\partial \theta_k}(\boldtheta) \Pi_j(\boldtheta) \, d\boldtheta,
\end{align*}
where we applied dominated convergence to deduce the final line. By Lemma~\ref{lem:eigenvalues}, $\partial E_j/\partial\theta_k\neq 0$ a.e., so $\ker(Q_k) =\{0\}$.
\end{proof}

\bibliographystyle{amsplain}
\bibliography{BZbib}

\end{document}